\documentclass{article}
\usepackage{epsfig}
\usepackage{amsmath, amssymb, amsthm}	
\usepackage{color}			

\newcommand{\R}{{I\!\!R}}

\def\R{{\rm I}\! {\rm R}}

\newtheorem{theorem}{Theorem}[section]

\newtheorem{remark}{Remark}[section]

\newtheorem{definition}{Definition}[section]

\newcommand{\OT}{\mathcal O}

\begin{document}

\pagestyle{headings}

\title{Iterative Splitting Methods: Almost Asymptotic Symplectic Integrator for Stochastic Nonlinear Schr\"odinger Equation}
\author{J\"urgen Geiser
\thanks{University of Greifswald, Institute of Physics, Felix-Hausdorff-Str. 6, D-17489 Greifswald, Germany, E-mail: jgeiser@uni-greifswald.de}}
\maketitle

\begin{abstract}

In this paper we present splitting methods which
are based on iterative schemes and
applied to stochastic nonlinear Schr\"odinger equation.
We will design stochastic integrators
which almost conserve the symplectic structure.

The idea is based on rewriting an iterative splitting
approach as a successive approximation method based on 
a contraction mapping principle and that we have an almost
symplectic scheme, see \cite{tan2005} and  \cite{jiang2013}.

We apply a stochastic differential equation, that we can 
decouple into a deterministic and stochatic part, while
each part can be solved analytically.
Such decompositions allow accelerating the methods and preserving, under
suitable conditions, the symplecticity of the schemes.

A numerical analysis and application to the stochastic Schr\"odunger equation are presented.

\end{abstract}




\section{Introduction}

The motivation is to develop fast solver schemes to
solve stochastic Hamiltonians in solitary waves and collisions.

The idea is based on almost asymptotic symplecticity for
stochastic Hamiltonian partial differential equations,
such underlying algorithms are applied to develop 
stochastic symplectic methods 
for solving a stochastic Schroedinger equations, see \cite{tan2005}. 

It is shown that the noval schemes preserve the symplectic structure
in an asymptotic regime, which means it is  $\OT(\delta^{n+1})$ away from
a symplectic scheme with $\delta \in (0,1)$. 

\begin{definition}
We consider a Hamiltonian system, while $u = (p, q)$ and we write:
\begin{eqnarray}
\label{sympl_1}
\frac{\partial u}{\partial t} = f(u(t)) = J \nabla_u H(u(t)) ,
\end{eqnarray}
where $(p, q) \in \R^d \times \R^d$ and $J = \left( \begin{array}{c c} 0 & - I_d \\ I_d & 0 \end{array} \right)$ and $I_d$ is the d-dimensional identity matrix,
$\nabla_u$ is the gradient with respect to $u$.

We assume that $\phi_{\tau}$ is the solution operator with $u(t^{n+1}) = \phi\tau(u^n)$, where $\tau$ isthe time step and we have the following definition about the symplecticity:
\begin{itemize}
\item $\phi_{\tau}$ preserves the symplecticness of the system (\ref{sympl_1}), if:
\begin{eqnarray}
\label{sympl_2}
(\frac{\partial \phi_{\tau}}{\partial z(t)}^T|_{t=t^n} J(\frac{\partial \phi_{\tau}}{\partial z(t)}|_{t=t^n} - J = 0,
\end{eqnarray}
\item $\phi_{\tau}$ preserves the almost (or asymptotic) symplecticness of the system (\ref{sympl_1}), if:
\begin{eqnarray}
\label{sympl_3}
|| (\frac{\partial \phi_{\tau}}{\partial z(t)}^T|_{t=t^n} J(\frac{\partial \phi_{\tau}}{\partial z(t)}|_{t=t^n} - J || \le C \delta^{m+1},
\end{eqnarray}
where $C$ is a constant with $\tau = \tilde{\tau}(\delta)$ and $\tilde{\tau}$ is a function of $\delta$, which is given from the solution method. 
\end{itemize}

\end{definition}

\begin{remark}
The idea of almost symplecticity has the origin of 
modifying the definition of symplecticity.
For example, if one assume that $J$ depends on $u$, then
one can proof, that we have an almost poisson structure
and we preserve the poisson structure up to the second order,
see \cite{austin1993}.
Such ideas are also used in the development in pseudo-symplectic methods, see \cite{aubry1998}.
\end{remark}

In the following, we deal with the stochastic nonlinear Schr\"odinger equation with multiplicative noise, which is given by
\begin{eqnarray}
\label{orig_1}
i \frac{\partial u}{\partial t} = \frac{\partial^2 u}{\partial x^2} + 2 |u|^{2 \sigma} u + \epsilon u \circ \frac{d W}{d t}, \; t > 0, \epsilon > 0, x \in \R,
\end{eqnarray}
where $u = u(x,t)$ is the complex-valued solution and $\circ$ denotes 
$\frac{d W}{dt}$ is defined as a real-valued white noise which is delta correlated in time and either smooth or delta correlated in space.

The deterministic nonlinear Schroedinger equation is given by
\begin{eqnarray}
i \frac{\partial u}{\partial t} = \frac{\partial^2 u}{\partial x^2} + 2|u|^{2 \sigma} u,
\end{eqnarray}
which is well-known in the literature \cite{schober1999}.

\section[Iterative Splitting]{Iterative Splitting as a Successive Approximation Method}

We can rewrite this to a Hamiltonian system by 

$u = p + i q$, where $p$ and $q$ are real-valued functions and we can separate it into the following form and we obtain a multi-symplectic system:
\begin{eqnarray}
&& \left(
\begin{array}{c}
d P \\
d Q
\end{array} \right) = 
\left(
\begin{array}{c}
- H_q(P, Q) dt - G_q(P, Q) \circ dW(t) \\
 H_p(P, Q) dt + G_p(P, Q) \circ dW(t)
\end{array}
\right)
\end{eqnarray}
where we have the symplectic structure $dP \wedge dQ = dp \wedge dq$.

The system is given by
\begin{eqnarray}
 \left(
\begin{array}{c}
d p  \\
d q
\end{array} \right) && = 
\left(
\begin{array}{c c}
0 &  A_1 + A_2(p, q) \\
- A_1 + A_2(p, q) & 0 
\end{array}
\right)
\left(
\begin{array}{c}
p \\
q
\end{array} \right) dt \nonumber \\
&&
+ \epsilon
\left(
\begin{array}{c c}
0 &  I \\
- I & 0 
\end{array}
\right)
\left(
\begin{array}{c}
p \\
q
\end{array} \right)  \circ dW ,
\end{eqnarray}
where the matrices are given by the semi-discretization of the original system
(\ref{orig_1}).

\begin{theorem}
The iterative splitting scheme is almost symplectic.
\end{theorem}

\begin{proof}

For the Hamiltonian system
\begin{eqnarray}
d y = F(y) dt + G(y) dW
\end{eqnarray}
we apply the successive approximation method:
\begin{eqnarray}
y^{n+1, i+1} = K(y^n, y^{n+1,i}, y^{n+1,i}) = y^n + F(y^{n+1, i+1}) \Delta t + G(y^{n+1, i}) \Delta W , 
\end{eqnarray}
where we apply the linearised scheme:
\begin{eqnarray}
y^{n+1, i+1} = \tilde{K}(y^n, y^{n+1,i+1}, y^{n+1,i}) = y^n + \tilde{F} y^{n+1, i+1} \Delta t + \tilde{G} y^{n+1, i} \Delta W , 
\end{eqnarray}
further, the contraction mapping is given by
\begin{eqnarray}
&& ||  \tilde{K}(y^n, y^{n+1,i+1}, y^{n+1,i}) - \tilde{K}(x^n, x^{n+1,i+1}, x^{n+1,i}) || \nonumber \\
&& \le  \tilde{\rho} || y^{n+1, i+1} - x^{n+1, i+1} || , 
\end{eqnarray}
where $\tilde{\rho} = \rho_1 + \rho_2$ and
$\rho_1 = \Delta t ||\tilde{F}||$ and $\rho_2 = \Delta W ||\tilde{G}||$.
\end{proof}

\section{Almost Symplectic Scheme}
\label{symp}

In the following, we discuss the linearised equation
in the algorithm.

We have the fixed-splitting discretisation step-size $\tau$, on the
time-interval $[t^n,t^{n+1}]$, and the stochastic time step $\Delta W = W_{t^{n+1}} - W_{t^n} = \Delta t X$ (Wiener process), where $X$ is a Gaussian distributed  
random variable with $E(X)=0$ and $Var(X) = 1$, see \cite{kloeden1992}.

We solve the following sub-problems
consecutively for $i=1,2, \dots ,m+1$. (cf. \cite{gei_2009_5}):
\begin{eqnarray}
 && d y_i(t) = A y_i(t) dt \; + \; B y_{i-1} dW_t(t), \;
\mbox{with} \; \; y_i(t^n) = y^{n} \label{kap3_iter_1} \\
&& \mbox{and} \; y_{0}(t^n) = y^n \; , \; y_{-1} = 0.0  \nonumber
\end{eqnarray}
 where $y^n$ is the known successive approximation at the 
time-level $t=t^{n}$. The split approximation at the time-level $t=t^{n+1}$ 
is defined by $y^{n+1}=y_{m+1}(t^{n+1})$.

We can rewrite this into the following ODE form:
\begin{eqnarray}
 && \frac{\partial y_i(t)}{\partial t} =  A y_i(t) \; + \; B y_{i-1} \dot{W}_t , \;
\mbox{with} \; \; y_i(t^n) = y^{n} \label{kap3_iter_1} \\
&& \mbox{and} \; y_{0}(t^n) = y^n \; , \; y_{-1} = 0.0 , \nonumber
\end{eqnarray}
where $\dot{W}_t = \frac{d W_t}{d t}$.

\smallskip
\begin{theorem}{\label{Th1}}
We are given $A,B \in {\mathcal L(\Omega)} $ linear bounded operators
(e.g., due to the linearisation) and we consider the abstract Cauchy problem
\begin{equation}\label{eq:ACP}
\begin{array}{c}
 d y(t) = A y dt  + B y dW_t, \quad 0 < t \leq T \\
\noalign{\vskip 1ex} {\displaystyle y(0)=y_0. }
\end{array}
\end{equation}
\noindent
 Then the  problem (\ref{eq:ACP}) has a unique solution;  the iterations
(\ref{kap3_iter_1}) over  $i=1,2, \dots, m+1$ are convergent with order ${\mathcal O} (\sqrt{\Delta t}^{m+1})$.
\end{theorem}

\begin{proof}

The problem (\ref{eq:ACP}) has a unique solution $c(t) = \exp((A \; \Delta t + B \; \Delta W) \; c_0 $.

For the local error function $e_i(t)= y(t) - y_i(t)$, we have the relations
\begin{equation}
\begin{array}{c}
d e_i(t) = A e_i dt  + B e_{i-1} dW_t, \quad t \in (t^n,t^{n+1}], \\
\noalign{\vskip 1ex} {\displaystyle e_i(t^n)=0}.
\end{array} \label{eq:err1}
\end{equation}
Applying the method of variation of constants, 
the solution of the abstract Cauchy problem can be written as
\begin{equation}
\begin{array}{c}
{\displaystyle e_i(t) = \int_{t^n}^{t}{\exp (A (t-s)) B e_i(s)dW_s}, \quad t \in [t^n, t^{n+1}].}
\end{array} \label{eq:errsol}
\end{equation}
Furthermore, we have
\begin{equation}\label{eq:estim}
\begin{array}{c}
\| e_i\|(t) \le  \| B \|
\|{ e}_{i-1}  \|\int_{t^n}^{t}\|{\exp (A(t-s))\| dW_s},
\quad t \in [t^n, t^{n+1}].
\end{array}
\end{equation}
Based on our assumption that $A$ is bounded, we have
\begin{equation}
\begin{array}{c}
 \| e_i\|(t) \le  K \| B \| \sqrt{\Delta t} \| e_{i-1}\| , \quad t \in [t^n,t^{n+1}] .
\end{array} \label{eq:err3}
\end{equation}
where $|| \exp(A t) || \le K, \; t > 0 $.

The estimations (\ref{eq:err3}) result in
\begin{equation}
\begin{array}{c}
\| e_{m+1}\| = K_1 \sqrt{\Delta t}^{m+1} \|{e}_{0}\| + {\mathcal O }(\sqrt{\Delta t}^{m+2}),
\end{array} \label{eq:basic3}
\end{equation}
which proves our statement.
\end{proof}

Furthermore, the almost asymptotic symplecticity of the scheme (\ref{kap3_iter_1}) 
is given as:

\smallskip
\begin{theorem}{\label{Th2}}

Consider the algorithm (\ref{kap3_iter_1}) and let $\phi_{\Delta t}^i$ be the 
solver step of the algorithm.
Then for any $\delta \in (0, 1)$, there exists $\sqrt{\tau} \le \frac{\delta}{K_1}$, where $K_1 = || B ||^{m+1}$ and the time-step $\Delta t \le \tau$, where $m$ is the number of iterative steps,
and we have
\begin{equation}\label{eq_sympl}
\begin{array}{c}
(\frac{\partial \phi_{\Delta t}^i}{\partial y_{0}})^t J (\frac{\partial \phi_{\Delta t}^i}{\partial y_{0}}) - J || \le C \delta^{m+1} , \forall y_{0} \in \Omega ,\\
\end{array}
\end{equation}
where $C$ is a constant.

\end{theorem}

\begin{proof}

The algorithm (\ref{kap3_iter_1}) has the following 
solution:
\begin{equation}
\begin{array}{c}
\phi_{\Delta t}^i = \exp(A \; \Delta t) y_n + \int_{t^n}^{t^{n+1}} \exp(A (t^{n+1} - s)) \; B \; y_{i-1}(s) \; dW_s .
\end{array} \label{sol_sympl}
\end{equation}
Furthermore, we have
\begin{equation}
\begin{array}{c}
\| \frac{\partial \phi_{\Delta t}^i}{\partial y_{i}} \} \le || \exp(A (\Delta t)) \; B \; \Delta W_t || , \\
\le C || B || \; \sqrt{\Delta t} , 
\end{array}
\end{equation}
and the recursion is given by
\begin{equation}
\begin{array}{c}
\| \frac{\partial \phi_{\Delta t}^i}{\partial y_{0}} \} \le \tilde{C} \; \sqrt{\Delta t}^{i+1} , 
\end{array}
\end{equation}
when the estimations result in
\begin{equation}
\begin{array}{c}
\| (\frac{\partial \phi_{\Delta t}^m}{\partial y_{0}})^t J (\frac{\partial \phi_{\Delta t}^m}{\partial y_{0}}) - J || \le C \delta^{m+1} , \forall y_{0} \in \Omega ,\\
\end{array}
\end{equation}
and $\sqrt{\tau} \le \frac{\delta}{K_1}$, which proves our statement.

\end{proof}

\section{Numerical Methods}

In the following, we treat the different numerical methods.

The underlying equation is given as
\begin{eqnarray}
i \frac{\partial u}{\partial t} = \lambda \frac{1}{2} \frac{\partial^2 u}{\partial x^2} + 2 |u|^{2 \sigma} u + \epsilon u \circ \frac{d W}{d t}, \; t > 0, \epsilon > 0, x \in \R,
\end{eqnarray}
where the initial values are given as $u_{t_0}= u_0$, $\lambda \in \R$
and $W$ is a Wiener process.

We apply a semi-discretisation via finite difference schemes and obtain the
ODE problem
\begin{eqnarray}
i \frac{\partial u}{\partial t} = A u + B(u) u + C u \circ \frac{d W}{d t}, \; t > 0, \epsilon > 0, x \in \R,
\end{eqnarray}
where the operators are given by
\begin{eqnarray}
A = \lambda \frac{1}{\Delta x^2} [1 \; -2 \; 1] , \\
B(u) = 2 |u|^{2 \sigma}  \\
C = \epsilon u ,
\end{eqnarray}
where we apply the different splitting schemes.

\subsection{Linearised stochastic Schroedinger equation}

We consider the following linearised stochastic Schroedinger equation:
\begin{eqnarray}
&& i \frac{\partial u}{\partial t} = - \frac{1}{2} \frac{\partial^2}{\partial x^2} u + V(x,t) u + \psi |u|^2 u  \nonumber \\
&& + \epsilon u \circ \frac{dW}{dt}, (x,t) \in [0, 1] \times [0, 1] , \\
&& u(x, 0) = u_0(x) , x \in [0,1] , \\
&& u(0, t) = u(1, t) , t \in [0, 1] ,
\end{eqnarray}
where $u_0(x) = \exp(\sin(2 x))$. \\
We assume periodic boundary conditions $u(x_L, t)= u(x_R, t)$, \\
where $\Omega = [x_L , x_R]$, e.g. $x_L = 0$, $x_R = 1.0$ and $\epsilon$ is small.

We employ the following transformation and change of variables:

$u = \eta + i \xi$ \\
\begin{eqnarray}
&& \left(
\begin{array}{c}
\dot{\eta} \\
\dot{\xi}
\end{array} \right) = 
\left(
\begin{array}{c c}
0 &  A(t,x, \eta, \xi) \\
- A(t,x, \eta, \xi) & 0 
\end{array}
\right)
\left(
\begin{array}{c}
\eta \\
\xi
\end{array} \right)
\end{eqnarray}

We apply a finite difference discretisation and the matrices are
given as
\begin{eqnarray}
&& A(t,x, \eta, \xi) = A_1(t,t) + A_2(t,x, \eta, \xi) + A_3(t,x) , \\
&& A_1(t,x) = - \frac{1}{2} \frac{1}{\Delta x^2} [1 \; -2 \; 1] , \\
&& A_2(t,x, \eta, \xi) = \left( V(x) + 2 (\eta^2 + \xi^2) \right) \;  [0 \; 1 \; 0 ], \\
&& A_3(t,x) = \epsilon \Delta W  \;  [0 \; 1 \; 0 ].
\end{eqnarray}
\begin{eqnarray}
&& A_1(t,x) = - \frac{1}{2} \frac{1}{\Delta x^2} 
\left( \begin{array}{c c c c c c}
-2 & 1 & 0 & \ldots & 0 & 1  \\
1 & -2 & 1 & 0 & \ldots & 0 \\
0 & 1 & -2 & 1 & 0 & \ldots  \\
\vdots & \ddots  & \ddots & \ddots & \ddots & \vdots \\
1 & 0 & \ldots & 0 &  1 & -2 \\ 
\end{array} \right) , \\
&& A_2(t,x, \eta, \xi) = \left( \begin{array}{c c c c c c}
 \tilde{V}(x_{1})  & 0 & 0 & \ldots & 0 & 0  \\
0 &  \tilde{V}(x_{2})  & 0 & 0 & \ldots & 0  \\
0 & 0 &  \tilde{V}(x_{3})  & 0 & 0 & \ldots  \\
\vdots & \ddots  & \ddots & \ddots & \ddots & \vdots \\
0 & 0  & \ldots & \ldots & 0& \tilde{V}(x_{M}) \\
\end{array} \right) , \\
&& \tilde{V}(x_i) = V(x_i) + 2 (\eta^2(x_i, t^{n-1}) + \xi^2(x_i, t^{n-1})) , \\
&& A_3(t,x) = \left( \begin{array}{c c c c c c}
 \epsilon  & 0 & 0 & \ldots & 0 & 0  \\
0 &   \epsilon & 0 & 0 & \ldots & 0  \\
0 & 0 &   \epsilon & 0 & 0 & \ldots  \\
\vdots & \ddots  & \ddots & \ddots & \ddots & \vdots \\
0 & 0  & \ldots & \ldots & 0&  \epsilon \\
\end{array} \right) , 
\end{eqnarray}
where we have $V(x,t) = 1.0$, $\epsilon = 1$, $\Delta x = 0.1, 0.01. 0.001$.

We apply the operator splitting schemes:
\begin{eqnarray}
&& \left(
\begin{array}{c}
\eta^{n+1} \\
\xi^{n+1}
\end{array} \right) \nonumber \\
& = & \exp(\Delta t \; \tilde{A}_1) \; \exp(\Delta t \; \tilde{A}_2(\eta^{n}, \xi^{n})) \; \\
&& \cdot  \exp(- \frac{1}{2} \Delta t \; (\tilde{A}_3^t \tilde{A}_3) + \tilde{A}_3 \Delta W_t )  \left(
\begin{array}{c}
\eta^n \\
\xi^n
\end{array} \right) , 1 \le n \le N , \nonumber \\
&& \eta^0 = (\exp(\sin(2 x_1)), \ldots, \exp(\sin(2 x_M)))^t , \xi^0 = (0, \ldots , 0)^t
\end{eqnarray}
where $\Delta t = t^{n+1} - t^n$, the random variable $W_t$ is based on a Wiener process with $\Delta W_t = W_{t^{n+1}} - W_{t^n} = \sqrt{\Delta t} X$, and
$X$ is a Gaussian distributed random variable with $E(X) = 0$ and 
$Var(X) = 1$. This means we have $\Delta W_t = rand \sqrt{\Delta t}$.

The splitting operators are
\begin{eqnarray}
\tilde{A}_1 =  \left(
\begin{array}{c c}
0 &  A_1(t,x, \eta, \xi) \\
- A_1(t,x, \eta, \xi) & 0 
\end{array}
\right)  \in \R^{2m \times 2m} , \\
\tilde{A}_2 =  \left(
\begin{array}{c c}
0 &  A_2(t,x, \eta, \xi) \\
- A_2(t,x, \eta, \xi) & 0 
\end{array}
\right)  \in \R^{2m \times 2m} , \\
\tilde{A}_3 =  \left(
\begin{array}{c c}
0 &  A_3(t,x, \eta, \xi) \\
- A_3(t,x, \eta, \xi) & 0 
\end{array}
\right) \in \R^{2m \times 2m}
\end{eqnarray}

We present the different convergent time-steps results
for $|u| = \sqrt{(\eta^2 + \xi^2)}$.

The analytical solution is

We apply the operator splitting schemes as follows:
\begin{eqnarray}
&& \left(
\begin{array}{c}
\eta^{n+1} \\
\xi^{n+1}
\end{array} \right) \nonumber \\
&&  = \exp( (\tilde{A}_1 + \tilde{A}_2(\eta^{n}, \xi^{n}) - \frac{1}{2} (\tilde{A}_3^t \tilde{A}_3)) \Delta t + \tilde{A}_3 \Delta W_t )  \left(
\begin{array}{c}
\eta^n \\
\xi^n
\end{array} \right) , 1 \le n \le N , \nonumber \\
&& \eta^0 = (\exp(\sin(2 x_1)), \ldots, \exp(\sin(2 x_M)))^t , \xi^0 = (0, \ldots , 0)^t
\end{eqnarray}
where $\Delta t = t^{n+1} - t^n$, the random variable $W_t$ is based on a Wiener process with $\Delta W_t = W_{t^{n+1}} - W_{t^n} = \sqrt{\Delta t} X$, and
$X$ is a Gaussian distributed random variable with $E(X) = 0$ and 
$Var(X) = 1$. This means we have $\Delta W_t = rand \sqrt{\Delta t}$.

The solution is given by $|u|$ and the errors are
\begin{eqnarray}
\label{ana_1}
&& || u_{refer}(x,t) - u_{i,j}(t)||_{L_2(0,T)} = \Delta x \; \sum_{n=1}^N (u_{refer}(x_i, t) - u(x_i,t))^2 , \\
&& E(||  u_{refer}(x, t) -  u_{i,j}(x, t)||) = \frac{1}{N} \; \sum_{n=1}^N | u_{refer}(x_i, t) - u(x_i,t) | ,
\end{eqnarray}

In the following figures, we present the results for the error
of the iterative splitting schemes, see Fig.~\ref{error_1}.
\begin{figure}[ht]
\begin{center}  
\includegraphics[width=5.0cm,angle=-0]{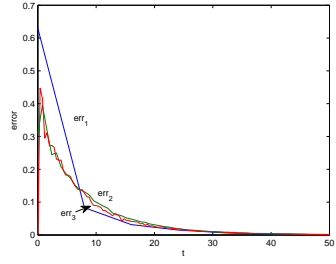}
\end{center}
\caption{\label{error_1} The $L_2$-errors of the iterative
splitting scheme}
\end{figure}

In the following figures, we present the results for the 
different splitting schemes, see Fig. \ref{experiment_2_1}.
\begin{figure}[ht]
\begin{center}  
\includegraphics[width=5.0cm,angle=-0]{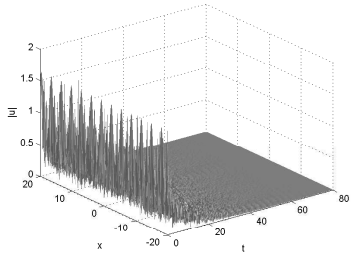}
\end{center}
\caption{\label{experiment_2_1} The results of the A--B splitting with $\Delta t = 0.005, \Delta x = 0.005$.}
\end{figure}

\begin{remark}
With more iterative steps, we see an improvement in the numerical results.
With two to three iterative steps, we obtain nearly the analytical solution.
Here, we could see the almost asymptotic behaviour of the scheme.
\end{remark}

\subsection{Deterministic Schroedinger equation: Perturbations}

We consider the following equation:

\begin{eqnarray}
&& i \hbar \frac{\partial u}{\partial t} = H u , \\
&& H u = - \frac{1}{2} \frac{\partial^2}{\partial x^2} u + \left( \frac{1}{1 + \sin^2(x)} + \lambda |u|^2 \right) u ,
\end{eqnarray}
where $\lambda = 30$, $u_0 = \exp(\sin(2x))$.

We employ the following transformation and change of variables:

$u = \eta + i \xi$ \\

\begin{eqnarray}
&& \left(
\begin{array}{c}
\dot{\eta} \\
\dot{\xi}
\end{array} \right) = 
\left(
\begin{array}{c c}
0 &  A(t,x, \eta, \xi) \\
- A(t,x, \eta, \xi) & 0 
\end{array}
\right)
\left(
\begin{array}{c}
\eta \\
\xi
\end{array} \right)
\\
&& A(t,x, \eta, \xi) = A_1(t,t) + A_2(t,x, \eta, \xi) , \\
&& A_1(t,x) = - \frac{1}{2} \frac{1}{\Delta x^2} [1 \; -2 \; 1] , \\
&& A_2(t,x, \eta, \xi) = \left( \frac{1}{1 + \sin^2(x)} + \lambda (\eta^2 + \xi^2) \right) \;  [0 \; 1 \; 0 ].
\end{eqnarray}

The underlying discretised matrices for the splitting schemes are given as:
\begin{eqnarray}
&& \tilde{A}_1 = 
\left(
\begin{array}{c c}
0 &  A_1(t,x) \\
- A_1(t,x) & 0 
\end{array}
\right) ,\\
&& \tilde{A}_2(t, x, \eta(t^{n-1},x), \xi(t^{n-1}),x) \\
&& = 
\left(
\begin{array}{c c}
0 &  A_2(t,x, \eta(t^{n-1},x), \xi(t^{n-1}),x) \\
- A_2(t,x, \eta(t^{n-1},x), \xi(t^{n-1}),x) & 0 
\end{array}
\right) . \nonumber
\end{eqnarray}

In the next list of schemes we discuss different splitting scheme.
The first splitting scheme is known as an A-B splitting or Lie-Trotter splitting scheme, see \cite{trotter1959}, while we apply multiplicative the different separated operators. The second splitting scheme is known as an iterative splitting scheme,
see \cite{geiser_2010}. Such a scheme apply iteratively the separated operators based on 
a fix-point approximation, see \cite{gei_2011}.

We will employ the following splitting schemes:

\begin{itemize}
\item A--B splitting
\begin{eqnarray}
u_n = \exp(t \tilde{A}_1) \exp(t \tilde{A}_2)  u_{n-1} , 1 \le n \le N ,
\end{eqnarray}
\item Strang splitting scheme
\begin{eqnarray}
u_n = \exp(t/2 \tilde{A}_1) \exp(t \tilde{A}_2)  \exp(t/2 \tilde{A}_1)  u_{n-1} , 1 \le n \le N ,
\end{eqnarray}
\item Weighted Iterative Splitting 1:
We define a relaxed iterative splitting method based on the critical value $\lambda$:
\begin{eqnarray}\label{iter1}
\dot{u_i} &= & (\tilde{A}_1 + (1 - \omega) \tilde{A}_2 ) u_i +  \omega \tilde{A}_2 u_{i-1} , \\
 &= & \hat{A}_1 u_i +  \hat{A}_2 u_{i-1} ,
\end{eqnarray}
and $\hat{A}_1 = \tilde{A}_1 + (1 - \omega) \tilde{A}_2$,  $\hat{A}_2 = \omega \tilde{A}_2$ and $\omega = \frac{1}{\lambda}$.

The algorithm is
\begin{eqnarray}\label{iter1}
\dot{u_1} &= & \hat{A}_1 u_1,\\
\label{iter2} \dot{u_2} &= & \hat{A}_1 u_2 + \hat{A}_2
u_1,\\\label{iter3}
\dot{u_3} &= & \hat{A}_1 u_3+ \hat{A}_2 u_2,\\
\dot{u_4} &= & ...
\end{eqnarray}
and is solved as:
\begin{eqnarray}
&& c_1(t) = \exp(\hat{A}_1t) c(t^n), \\
&& c_2(t) = c_1(t) +  c_1(t) \; \int_0^t [\hat{A}_2, \exp(s \hat{A}_1)] ds \; , \nonumber \\
&& c_2(t) = c_1(t) +  c_1(t) \; [\hat{A}_2, \phi_1(t \hat{A}_1)] ,  \\
&& c_3(t) = c_2(t) + c_1(t) \; \int_0^t [\hat{A}_2, \exp(s \hat{A}_1)]  [\hat{A}_2, \phi_1(s \hat{A}_1)]  ds \;  ,  \nonumber \\
&& c_3(t) = c_2(t) + c_1(t) \;\left(  [\hat{A}_2, \exp(t \hat{A}_1)]  [\hat{A}_2, \phi_2(t \hat{A}_1)] \right.  \\
&& \left. +  [\hat{A}_2, \hat{A}_1 \exp(t \hat{A}_1)]  [\hat{A}_2, \phi_3(t \hat{A}_1)] \right) + O(t^3) \; , \nonumber \\
&& \cdots \nonumber
\end{eqnarray}
where the given $\phi_{i}$ is defined as:
\begin{eqnarray}
\phi_{0}(\hat{A}_1 t) =  \exp(\hat{A}_1 t) , \\
\phi_{i}(\hat{A}_1 t) = \int_0^t \phi_{i-1}(\hat{A}_1 s) \; ds , \\
\phi_{i}(\hat{A}_1 t) = \frac{\phi_{i-1}(\hat{A}_1 t) - I \; \frac{t^{i-1}}{(i-1)!}}{\hat{A}_1} .
\end{eqnarray}

\item Weighted Iterative Splitting 2:
We define a relaxed iterative splitting method based on the critical value $\lambda$:
\begin{eqnarray}\label{iter1}
&& \dot{u_i} = \tilde{A}_1 u_i +  \omega \tilde{A}_2 u_{i-1} , \\
&& \mbox{with} \; \; u_i(t^n) = u^{n}, \\
&& \mbox{and} \; u_{0}(t^n) = u^n \; , \; u_{-1} = 0 , \nonumber \\
&& \dot{u}_{i+1} = \omega \;\tilde{A}_1 u_i(t) \; + \; \tilde{A}_2 u_{i+1}(t), \; \\
&& \mbox{with} \; \; u_{i+1}(t^n) = \omega \; u^{n} + (1 - \omega)
\; u_i(t^{n+1}) \; , \nonumber
\end{eqnarray}
where $u^n$ is the known split
approximation at the  time level  $t=t^{n}$. The split
approximation at the time level $t=t^{n+1}$ is defined as
$u^{n+1}=u_{2m+1}(t^{n+1})$. The parameter $\omega \in [0,1]$. For
    $\omega = 0$, we have the sequential splitting
method, and for $\omega = 1$ we have the iterative splitting method.

\end{itemize}

The following figures present the results for the 
different splitting schemes, see Fig.~\ref{experiment_1}.
\begin{figure}[ht]
\begin{center}  
\includegraphics[width=5.0cm,angle=-0]{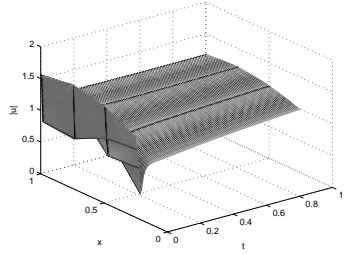}
\end{center}
\caption{\label{experiment_1} Results of the iterative splitting approach.}
\end{figure}

\begin{remark}
Here, we have compared the standard splitting scheme with our iterative
splitting approach. Based on the resolution of the analytical solution,
we obtain the same results as for the standard schemes.
\end{remark}

\subsection{Deterministic nonlinear Schr\"odinger equation}

We consider the equation
\begin{eqnarray}
&& i \frac{\partial u}{\partial t} = H u + \epsilon u \circ \frac{d W}{d t}, \; t > 0, \epsilon > 0, x \in \R , \\
&& H u = \left( \frac{\partial^2}{\partial x^2} + 2 |u|^{2 \sigma} \right) u ,
\end{eqnarray}
with $\sigma = 1.0$ and $\epsilon = 0.0$.

We choose the initial condition:
\begin{eqnarray}
u|_{t=0} = \frac{1}{\sqrt{2}} \sec(\frac{1}{\sqrt{2}}(x -25)) \exp(- i \frac{x}{20})  .
\end{eqnarray}
Then the exact single-soliton solution is
\begin{eqnarray}
u(x,t) = \frac{1}{\sqrt{2}} \sec(\frac{1}{\sqrt{2}}(x - \frac{t}{10} - 25)) \exp(- i (\frac{x}{20} + \frac{199}{400}t))  .
\end{eqnarray}

We employ the following transformation and change of variables:

$u = \eta + i \xi$ \\

\begin{eqnarray}
&& \left(
\begin{array}{c}
\dot{\eta} \\
\dot{\xi}
\end{array} \right) = 
\left(
\begin{array}{c c}
0 &  A(t,x, \eta, \xi) \\
- A(t,x, \eta, \xi) & 0 
\end{array}
\right)
\left(
\begin{array}{c}
\eta \\
\xi
\end{array} \right)
\\
&& A(t,x, \eta, \xi) = A_1(t,x) + A_2(t,x, \eta, \xi) , \\
&& A_1(t,x) = \frac{1}{\Delta x^2} [1 \; -2 \; 1] , \\
&& A_2(t,x, \eta, \xi) = \left( 2 (\eta^2 + \xi^2)^{\sigma} \right) \;  [0 \; 1 \; 0 ].
\end{eqnarray}

The underlying discretised matrices for the splitting schemes are
\begin{eqnarray}
&& \tilde{A}_1 = 
\left(
\begin{array}{c c}
0 &  A_1(t,x) \\
- A_1(t,x) & 0 
\end{array}
\right) ,\\
&& \tilde{A}_2(t, x, \eta(t^{n-1},x), \xi(t^{n-1}),x) \\
&& = 
\left(
\begin{array}{c c}
0 &  A_2(t,x, \eta(t^{n-1},x), \xi(t^{n-1},x)) \\
- A_2(t,x, \eta(t^{n-1},x), \xi(t^{n-1}),x) & 0 
\end{array}
\right) . \nonumber
\end{eqnarray}

We consider the following splitting schemes:

\begin{itemize}
\item A--B splitting
\begin{eqnarray}
u_n = \exp(t \tilde{A}_1) \exp(t \tilde{A}_2)  u_{n-1} , 1 \le n \le N ,
\end{eqnarray}
\item Strang splitting scheme
\begin{eqnarray}
u_n = \exp(t/2 \tilde{A}_1) \exp(t \tilde{A}_2)  \exp(t/2 \tilde{A}_1)  u_{n-1} , 1 \le n \le N ,
\end{eqnarray}
\item Weighted Iterative Splitting 1:
We define a relaxed iterative splitting method based on the critical value $\lambda$:
\begin{eqnarray}\label{iter1}
\dot{u_i} &= & (\tilde{A}_1 + (1 - \omega) \tilde{A}_2 ) u_i +  \omega \tilde{A}_2 u_{i-1} , \\
 &= & \hat{A}_1 u_i +  \hat{A}_2 u_{i-1} ,
\end{eqnarray}
and $\hat{A}_1 = \tilde{A}_1 + (1 - \omega) \tilde{A}_2$,  $\hat{A}_2 = \omega \tilde{A}_2$ and $\omega = \frac{1}{\lambda}$.

The algorithm is
\begin{eqnarray}\label{iter1}
\dot{u_1} &= & \hat{A}_1 u_1,\\
\label{iter2} \dot{u_2} &= & \hat{A}_1 u_2 + \hat{A}_2
u_1,\\\label{iter3}
\dot{u_3} &= & \hat{A}_1 u_3+ \hat{A}_2 u_2,\\
\dot{u_4} &= & ...
\end{eqnarray}
and is solved as
\begin{eqnarray}
&& c_1(t) = \exp(\hat{A}_1t) c(t^n), \\
&& c_2(t) = c_1(t) +  c_1(t) \; \int_0^t [\hat{A}_2, \exp(s \hat{A}_1)] ds \; , \nonumber \\
&& c_2(t) \approx c_1(t) +  c_1(t) \; t [\hat{A}_2, \exp(t \hat{A}_1)] , 
\end{eqnarray}
\end{itemize}

The following figures present the results for the 
different splitting schemes, see Fig.~\ref{experiment_2}.
\begin{figure}[ht]
\begin{center}  
\includegraphics[width=5.0cm,angle=-0]{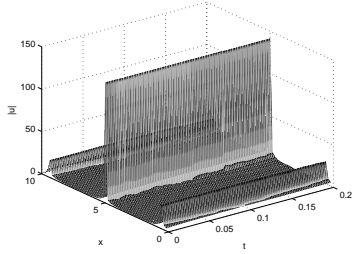}
\end{center}
\caption{\label{experiment_2} Results of the iterative splitting approach.}
\end{figure}

We apply $|u|$ for each solution and obtain the following errors:
\begin{eqnarray}
\label{ana_1}
&& || u_{refer}(x,t) - u_{i,j}(t)||_{L_2(0,T)} = \Delta x \; \sum_{n=1}^N (u_{refer}(x_i, t) - u(x_i,t))^2 , \\
&& E(||  u_{refer}(x, t) -  u_{i,j}(x, t)||) = \frac{1}{N} \; \sum_{n=1}^N | u_{refer}(x_i, t) - u(x_i,t) | ,
\end{eqnarray}

The following figures present the results for the errors
of the iterative splitting schemes, see Fig.~\ref{error_3}.

\begin{figure}[ht]
\begin{center}  
\includegraphics[width=5.0cm,angle=-0]{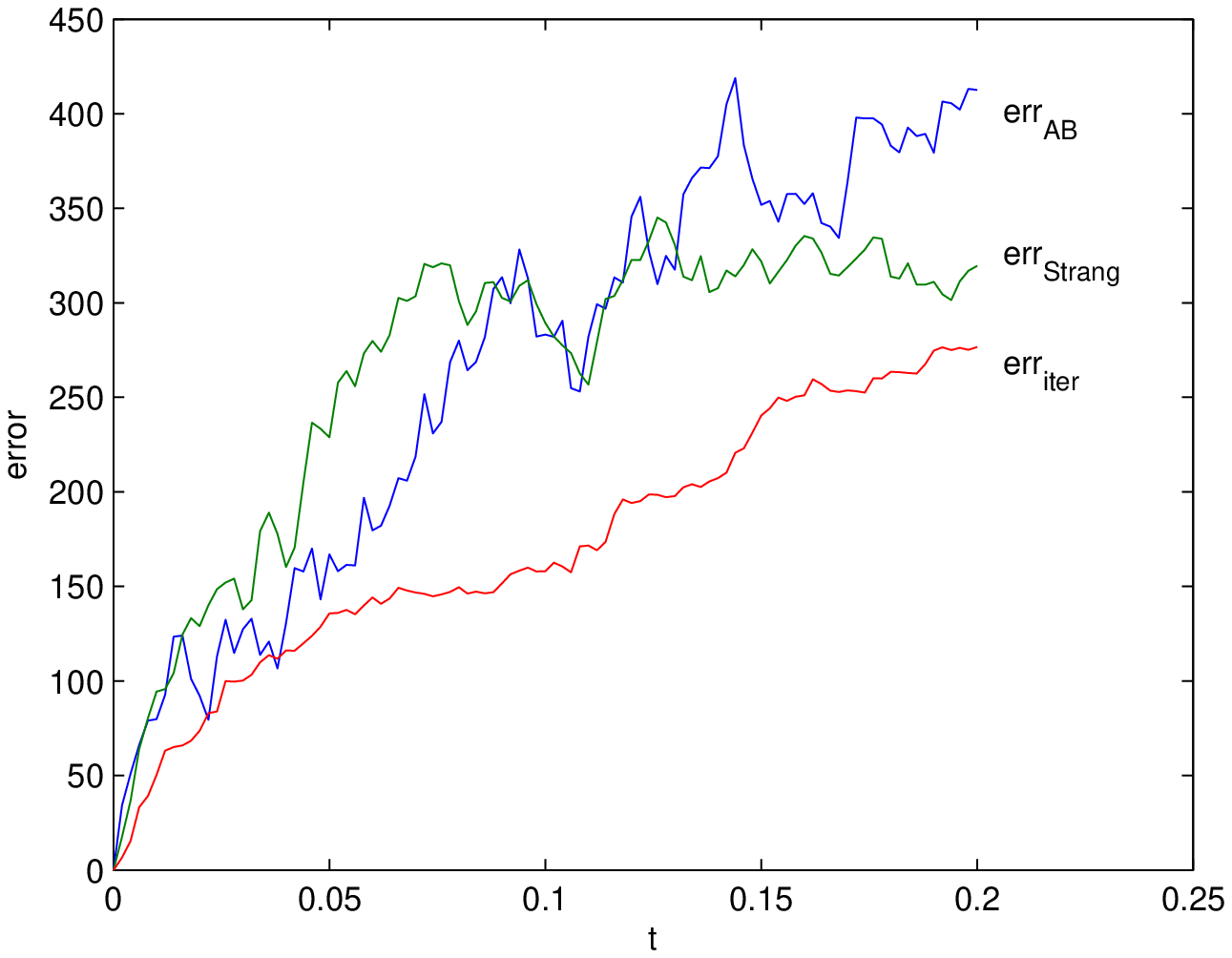}
\includegraphics[width=5.0cm,angle=-0]{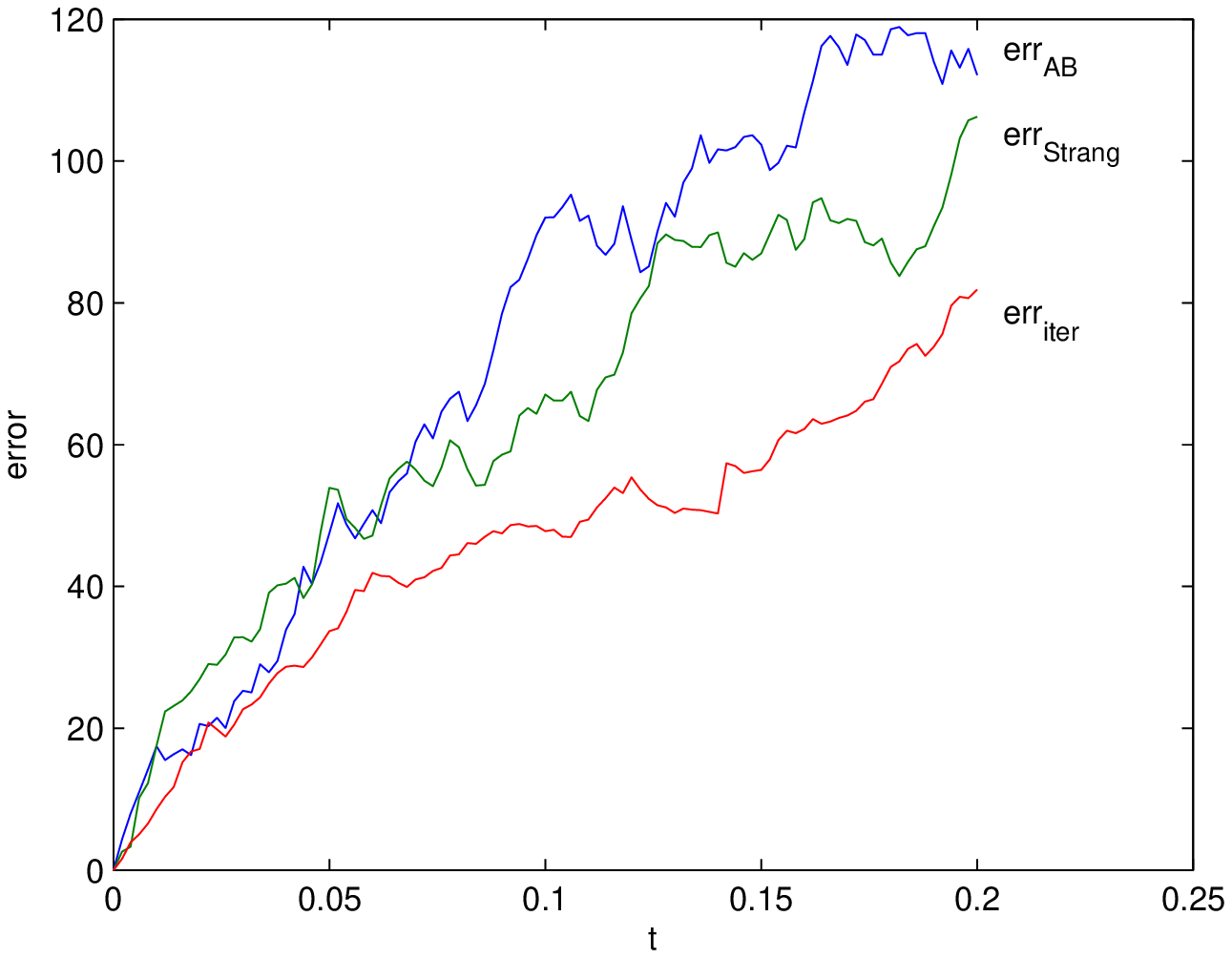} 
\end{center}
\caption{\label{error_3} The $L_2$-errors of the different splitting schemes, where in the left figure, we have $\Delta t = 0.002$ and $\Delta x = 0.01$ and in the right figure, we have $\Delta t = 0.002$ and $\Delta x = 0.02$.}
\end{figure}

\begin{remark}
In both resolution in time and space the iterative splitting method is 
more accurate than the standard A--B and Strang splitting schemes.
Here, we see an improvement based on the successive approximation idea
and obtain a more accurate linearisation than for the standard schemes.
\end{remark}

\subsection{Stochastic nonlinear Schr\"odinger equation}

We consider the equation
\begin{eqnarray}
&& i \frac{\partial u}{\partial t} = H u + \epsilon u \circ \frac{d W}{d t}, \; t > 0, \epsilon > 0, x \in \R , \\
&& H u = \left( \frac{\partial^2}{\partial x^2} + 2 |u|^{2 \sigma} \right) u ,
\end{eqnarray}
with $\sigma = 1.0$ and $\epsilon = 0.0$.

We choose the initial condition
\begin{eqnarray}
u|_{t=0} = \frac{1}{\sqrt{2}} \sec(\frac{1}{\sqrt{2}}(x -25)) \exp(- i \frac{x}{20})  .
\end{eqnarray}
%
%
%

For the reference solution, we apply a fine resolution Strang splitting.

We employ the following transformation and change of variables:

$u = \eta + i \xi$ \\

\begin{eqnarray}
&& \left(
\begin{array}{c}
\dot{\eta} \\
\dot{\xi}
\end{array} \right) = 
\left(
\begin{array}{c c}
0 &  A(t,x, \eta, \xi) \\
- A(t,x, \eta, \xi) & 0 
\end{array}
\right)
\left(
\begin{array}{c}
\eta \\
\xi
\end{array} \right)
\\
&& A(t,x, \eta, \xi) = A_1(t,x) + A_2(t,x, \eta, \xi) + A_3(t,x) , \\
&& A_1(t,x) = \frac{1}{\Delta x^2} [1 \; -2 \; 1] , \\
&& A_2(t,x, \eta, \xi) = \left( 2 (\eta^2 + \xi^2)^{\sigma} \right) \;  [0 \; 1 \; 0 ] , \\
&& A_3(t,x) = \epsilon \; \Delta W \; [0 \; 1 \; 0] , \\
\end{eqnarray}

The underlying discretised matrices for the splitting schemes are
\begin{eqnarray}
&& \tilde{A}_1 = 
\left(
\begin{array}{c c}
0 &  A_1(t,x) \\
- A_1(t,x) & 0 
\end{array}
\right) ,\\
&& \tilde{A}_2(t, x, \eta(t^{n-1},x), \xi(t^{n-1}),x) \\
&& = 
\left(
\begin{array}{c c}
0 &  A_2(t,x, \eta(t^{n-1},x), \xi(t^{n-1},x)) \\
- A_2(t,x, \eta(t^{n-1},x), \xi(t^{n-1}),x) & 0 
\end{array}
\right) , \nonumber \\
&& \tilde{A}_3 = 
\left(
\begin{array}{c c}
0 &  A_3(t,x) \\
- A_3(t,x) & 0 
\end{array}
\right) ,
\end{eqnarray}
and 
\begin{eqnarray}
&& \tilde{A}_4 = \tilde{A}_1 + \tilde{A}_2 .
\end{eqnarray}

\begin{itemize}
\item A--B splitting

We apply the operator splitting schemes as:
\begin{eqnarray}
&& \left(
\begin{array}{c}
\eta^{n+1} \\
\xi^{n+1}
\end{array} \right) \nonumber \\
&&  = \exp(\Delta t \; \tilde{A}_1) \; \exp(\Delta t \; \tilde{A}_2(\eta^{n}, \xi^{n})) \;  \\
&& \exp(- \frac{1}{2} \Delta t \; (\tilde{A}_3^t \tilde{A}_3) + \tilde{A}_3 \Delta W_t )  \left(
\begin{array}{c}
\eta^n \\
\xi^n
\end{array} \right) , 1 \le n \le N , \nonumber \\
&& \eta^0(x_i) = \frac{1}{\sqrt{2}} \sec(\frac{1}{\sqrt{2}}(x_i -25)) \cos(- \frac{x_i}{20}), i = 1, \ldots, M , \\
&& xi^0(x_i) = \frac{1}{\sqrt{2}} \sec(\frac{1}{\sqrt{2}}(x_i - 25)) \sin(- \frac{x_i}{20}) , i = 1, \ldots, M ,
\end{eqnarray}
where $\Delta t = t^{n+1} - t^n$, the random variable $W_t$ is based on a Wiener process with $\Delta W_t = W_{t^{n+1}} - W_{t^n} = \sqrt{\Delta t} X$, and
$X$ is a Gaussian distributed random variable with $E(X) = 0$ and 
$Var(X) = 1$. This means we have $\Delta W_t = rand \sqrt{\Delta t}$.
\item Iterative splitting scheme:

First iterative step
\begin{eqnarray}
&& X_{1,n}(t) =  \left(
\begin{array}{c}
\eta^{n+1} \\
\xi^{n+1}
\end{array} \right)  = \exp(\Delta t \; \tilde{A}_4) \left(
\begin{array}{c}
\eta^n \\
\xi^n
\end{array} \right) , \nonumber \\
\end{eqnarray}
Second iterative step
\begin{eqnarray}
&&  X_{2,n}(t) = X_{1,n}(t) + X_{1,n}(t) [\tilde{A}_3 , \int_0^t \exp(\tilde{A}_4 s) dW_s] ,  \quad t \in (t^n,t^{n+1}] , \nonumber \\
&& X_{2,n}(t) = X_{1,n}(t) + X_{1,n}(t) [\tilde{A}_3 , C_1(t)] ,  \quad t \in (t^n,t^{n+1}] , \nonumber
 \label{eq:err1}
\end{eqnarray}

The stochastic integral is computed as a Stratonovich integral:
\begin{eqnarray}
\label{int_1_1}
 && C_1(t) = \int_0^{t}  \exp(A s) dW_s \\
&& = \sum_{j=0}^{N-1}  \exp(A(\frac{t_{j} + t_{j+1}}{2})) \;  (W(t_{j+1}) - W(t_j)) , \nonumber \\
&& \Delta t = t /N, t_{j} = \Delta t + t_{j-1}, t_0 = 0 .
\end{eqnarray}
\end{itemize}

%

We apply $|u|$ for each solution and deal with the following errors:
\begin{eqnarray}
\label{ana_1}
&& || u_{refer}(x,t) - u_{i,j}(t)||_{L_2(0,T)} = \Delta x \; \sum_{n=1}^N (u_{refer}(x_i, t) - u(x_i,t))^2 , \\
&& E(||  u_{refer}(x, t) -  u_{i,j}(x, t)||) = \frac{1}{N} \; \sum_{n=1}^N | u_{refer}(x_i, t) - u(x_i,t) | .
\end{eqnarray}

The following figures present the results for the error
of the iterative splitting schemes, see Fig.~\ref{error_4}.
\begin{figure}[ht]
\begin{center}  
\includegraphics[width=5.0cm,angle=-0]{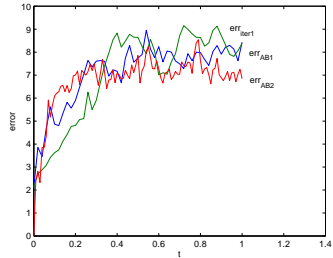} 
\end{center}
\caption{\label{error_4} The $L_2$-errors of the different splitting scheme, where we compare them to the solution obtained from a fine resolution iterative splitting scheme}
\end{figure}

\begin{remark}
In both resolution in time and space the iterative splitting method is 
more accurate than the standard A--B and Strang splitting schemes.
Here, we obtain an improvement based on the successive approximation 
scheme.
\end{remark}

\section{Conclusion}

We discuss the problems of using novel iterative splitting schemes
to solve stochastic nonlinear Schroedinger equations.
We could prove the almost asymptotic symplectic behaviour of the novel scheme.
The improvement with more iterative steps allows resolving the nonlinearity
and obtaining an improved symplectic scheme.
While standard splitting schemes have drawbacks as regards linearisation
and symplecticity, we could derive a combination of both higher accuracy and
conservation of the symplecticity.
In the future, we will take into account larger equation systems for 
a realistic application.

\bibliographystyle{plain}

\end{document}